\documentclass[11pt,oneside,reqno]{amsart}

\usepackage{hyperref}
\hypersetup{nesting=true,debug=true,naturalnames=true}
\usepackage{graphicx,amssymb}

\usepackage[square,sort&compress,comma,numbers]{natbib}
\usepackage{nicefrac,xcolor,upref,version}

\usepackage{amscd,amsthm,amsmath,amssymb,amsfonts}

\usepackage{mathrsfs}

\newtheorem{theorem}{Theorem}[section]
\newtheorem{lemma}[theorem]{Lemma}

\numberwithin{equation}{section}

\def\Q{{\mathbb {Q}}}

\def\N{{\bf N}} 
\def\Z{{\mathbb Z}}

\def\house#1{\setbox1=\hbox{$\,#1\,$}%
\dimen1=\ht1 \advance\dimen1 by 2pt \dimen2=\dp1 \advance\dimen2 by 2pt
\setbox1=\hbox{\vrule height\dimen1 depth\dimen2\box1\vrule}%
\setbox1=\vbox{\hrule\box1}%
\advance\dimen1 by .4pt \ht1=\dimen1
\advance\dimen2 by .4pt \dp1=\dimen2 \box1\relax}

  \def\eps{{\varepsilon}}

\def\build#1_#2^#3{\mathrel{\mathop{\kern 0pt#1}\limits_{#2}^{#3}}}

\def\date {le\ {\the\day}\ \ifcase\month\or
janvier\or fevrier\or mars\or avril\or mai\or juin\or juillet\or
ao\^ut\or septembre\or octobre\or novembre\or
d\'ecembre\fi\ {\oldstyle\the\year}}

\font\fivegoth=eufm5 \font\sevengoth=eufm7 \font\tengoth=eufm10

\newfam\gothfam \scriptscriptfont\gothfam=\fivegoth
\textfont\gothfam=\tengoth \scriptfont\gothfam=\sevengoth

\def \N{\mathbb{N}}

\def \Q{\mathbb{Q}}

\def \Z{\mathbb{Z}}

\begin{document}

\title{On inhomogeneous extension of Thue-Roth's type inequality with moving targets}
\author{Veekesh Kumar}

\address{Institute of Mathematical Sciences, HBNI, C.I.T Campus, Taramani, Chennai 600 113, India}
\email[Veekesh Kumar]{veekeshiitg@gmail.com}

\subjclass[2010] {Primary 11J68; Secondary 11J87 }
\keywords{Approximation to algebraic numbers, Pisot number, Schmidt Subspace Theorem .}
\begin{abstract} 
Let $\Gamma\subset \overline{\mathbb Q}^{\times}$ be a finitely generated multiplicative group of algebraic numbers.  Let $\delta, \beta\in\overline{\mathbb Q}^\times$  be  algebraic numbers with $\beta$ irrational.  In this paper,  we prove that    there exist only finitely many triples $(u, q,  p)\in\Gamma\times\mathbb{Z}^2$ 
with $d = [\mathbb{Q}(u):\mathbb{Q}]$  such that $|\delta qu|>1$ and 
$$
0<|\delta qu+\beta-p|<\frac{1}{H(u)^\varepsilon q^{d+\varepsilon}},
$$  
where $H(u)$ denotes  the absolute Weil height. This is an inhomogeneous analogue of the main theorem in  \cite{corv}.  As an application of our result, we also prove a transcendence result, which states as follows:  Let $\alpha>1$ be a real number. Let $\beta$ be an algebraic irrational number and  $\lambda$ be a non-zero real algebraic number.   For a given real number $\varepsilon >0$, if there are infinitely many natural numbers $n$ for which  $||\lambda\alpha^n+\beta|| < 2^{- \varepsilon n}$ holds true, then  $\alpha$ is transcendental, where $||x||$ denotes the distance from its nearest integer.
When $\alpha$ and $\beta$ both are algebraic numbers satisfying same conditions, then a particular result of Kulkarni, Mavraki and Nguyen in \cite{kul} asserts that $\alpha^d$ is a Pisot number. When $\beta $ is an algebraic irrational, our result implies that no algebraic number $\alpha$ satisfies the inequality for infinitely many natural numbers $n$. Also, our result  strengthens a result of Wagner and Ziegler \cite{wagner}.   
 \end{abstract}
\maketitle

\section{introduction}
For a real number $x$, let $||x||$  denote the distance of $x$ to its nearest integer, given by 
$$
||x||:=\mbox{min}\{|x-m|:m\in\mathbb{Z}\}.
$$

It is interesting to understand the behaviour of $||\alpha^n||$  for a given  real number $\alpha$ greater than 1.  In this context, in 1957 Mahler \cite{mahler} showed that for   $\alpha\in \mathbb{Q}\backslash\Z$ with $\alpha>1$ and $\varepsilon$  a positive real number, there are only finitely many $n\in \N$ satisfying  $||\alpha^n||<2^{-\varepsilon n}$. The key ingredient in Mahler's proof was the $p$-adic extension of Roth's theorem established by Ridout \cite{ridd1}. Mahler also asked for which algebraic number $\alpha$ the above conclusion holds true.
\smallskip

In 2004, by ingenious applications of the Subspace Theorem, Corvaja and Zannier \cite{corv} proved a `Thue-Roth' type inequality with `moving targets'. As an application of this result,  they answered  the  question of Mahler and  proved the following: {\it let $\alpha>1$ be a real algebraic number and let $\varepsilon$ be a positive real number. Suppose that $|| \alpha^n ||< 2^{- \varepsilon n}$  for infinitely many $n$. 
Then, there is some integer $d\geq 1$ such that the number $\alpha^d$ is a Pisot number. In particular $\alpha$ is an algebraic integer. }  We recall that a  real algebraic integer $\alpha>1$  is called a Pisot number, if the modulus value of all its   Galois conjugates other than $\alpha$ lie inside  the open unit disc.
\bigskip

In this paper, the main aim is to prove an inhomogeneous extension of Thue-Roth's type inequality with moving targets in the same spirit as the result of Corvaja and Zannier in \cite{corv}. We prove the following. 

\begin{theorem}\label{maintheorem}
Let $\Gamma\subset \overline{\mathbb Q}^{\times}$ be a finitely generated multiplicative group of algebraic numbers.  Let $\delta$ be a non-zero algebraic number, $\beta\in(0,1)$ be an algebraic  irrational, and $\varepsilon>0$ be a fixed  real number.  Then there exist only finitely many triples $(u, q,  p)\in\Gamma\times\mathbb{Z}^2$ with $d=[\mathbb{Q}(u):\mathbb{Q}]$   such that  $|\delta q u|>1$  and  
\begin{equation*}\label{eq1.1}
\tag{1.1}
0<|\delta qu+\beta-p|<\frac{1}{H(u)^\varepsilon q^{d+\varepsilon}}.
\end{equation*}
\end{theorem}

Recently in 2019, Kulkarni, Mavraki and Nguyen \cite{kul} generalized  Mahler's problem to an arbitrary linear recurrence sequence of the form $\{Q_1(n)\alpha_1^n+\cdots+Q_k(n)\alpha^n_k:n\in\N\}$, where $\alpha_i$'s are non-zero algebraic numbers and $Q_i(x)\in\overline{\mathbb{Q}}[x]\backslash \{0\}$. In a  particular case, they proved the  following inhomogeneous extension of the problem of Mahler: {\it let $\alpha>1$  be a real number, $\beta$ be a real algebraic number and let 
$\varepsilon$ be a positive real number.  
Supose that $||\alpha^n+\beta|| < 2^{- \varepsilon n}$  for infinitely many $n$. 
Then either $\alpha$ is transcendental or there is an integer $d\geq 1$ such that  $\alpha^d$ is a Pisot number. }
\bigskip

%

In the above result, if $\beta$ is an integer and $\alpha$ is an algebraic number such that  $\alpha^d$ is a Pisot number, then clearly there are infinitely many natural numbers  $n$ satisfying $||\alpha^{dn}+\beta||<2^{-\eps n}$ for some $\eps>0$. Thus, we can conclude that   the above assertion is best possible,  if $\beta$ is an integer.  
However, if $\beta$ is an algebraic irrational, as an application of Theorem \ref{maintheorem}, we deduce the following surprising result. 

\begin{theorem}\label{maintheorem2}
Let $\alpha>1$ be a real number. Let $\beta$ be an algebraic irrational and  $\lambda$ be a non-zero real algebraic number.   For a given real number $\varepsilon >0$, if there are infinitely many natural numbers $n$ for which  $||\lambda\alpha^n+\beta|| < 2^{- \varepsilon n}$ holds true, then  $\alpha$ is transcendental.
\end{theorem}
\smallskip

\smallskip

Note that Theorem \ref{maintheorem2} strengthens   the main result of Wagner and Ziegler \cite[Theorem 2]{wagner}.  
\bigskip

\section{Preliminaries} 

Let $K$ be a number field which is a Galois extension over $\mathbb{Q}$. Let $M_K$ be the set of all places on $K$ and $M_\infty$ be the set of all archimedean places on $K$. 
 For each place $w\in M_K$, let $K_w$ denote the completion of the number field $K$ with respect to $w$ and $d(w)=[K_w:\mathbb{Q}_\mathit{v}]$, where $\mathit{v}$ is the restriction of $w$ to $\mathbb{Q}$. 
For  every $w\in M_K$ whose restriction on $\mathbb{Q}$ is $v$ and $\alpha\in K$, we define the  normalized  absolute value $|\cdot|_w$ as follows:
\begin{equation*}\label{eq2.1}
\tag{2.1}
|\alpha|_w:=|\mbox{Norm}_{K_w/\mathbb{Q}_v}(\alpha)|^{\frac{1}{[K:\mathbb{Q}]}}_v.
\end{equation*}
Indeed if $w\in M_\infty$, then there exists an automorphism $\sigma\in\mbox{Gal}(K/\mathbb{Q})$ of $K$ such that for all $x\in K$, 
$$
|x|_w=|\sigma(x)|^{d(K)/[K:\mathbb{Q}]},
$$
where $d(K) =1$ if $K\subset \mathbb{R}$ and $d(K) = 2$  otherwise.  
\smallskip

 
\smallskip

Thus under the definition \eqref{eq2.1}, the    product formula   $\displaystyle\prod_{w\in M_K}|x|_w=1$ holds for any  $x\in K^\times$ and the 
absolute Weil  height  $H(x)$ is defined as 
$$
H(x):=\prod_{w\in M_K}\mbox{max}\{1,|x|_w\}.
$$
One can see that this height is independent of the choice of the number field $K$ containing $x$.
\bigskip

For a vector $\mathbf{x}=(x_1,\ldots,x_n)\in K^n$  and for a place $w\in M_K$, the $w$-norm for  $\mathbf{x}$ denoted by $||\mathbf{x}||_w$ is given by 
$$
||\mathbf{x}||_w:=\mbox{max}\{|x_1|_w,\ldots,|x_n|_w\}
$$
and  the projective height,  $H(\mathbf{x})$, is defined by 
$$
H(\mathbf{x}):=\prod_{w\in M_K}||\mathbf{x}||_w.
$$

For a finite set $S\subset M_K$ of places on $K$ which contains $M_\infty$,  the ring of $S$-integers, denoted by $\mathcal{O}_S$, is defined as 
$$
\mathcal{O}_S:=\mathcal{O}_{K,S}=\{\alpha\in K: |\alpha|_v\leq 1~~~~\mbox{for all}~~~v\notin S\}.
$$ 
The group of $S$-units in $K$, denoted by $\mathcal{O}^\times_S$ is the set of all  invertible elements of $\mathcal{O}_S$, defined as 
$$
\mathcal{O}^\times_S:=\{\alpha\in K:|\alpha|_v=1 ~~~\mbox{for all}~~~ v\notin S\}.
$$
\smallskip

Now we are ready to present a more general version  of the Schmidt Subspace Theorem, which was  formulated by Schlickewei and Evertse.  For the reference,  see (\cite[ Chapter 7]{bomb},  \cite[ Chapter V, Theorem 1D$^\prime$]{schmidt} and  \cite[Page 16, Theorem II.2]{zannier}).
\smallskip

\begin{theorem} (Schlickewei) \label{schli}
 Let $K$ be an algebraic number field and $m \geq 2$ an integer. Let $S$ be a finite set of places on $K$ containing all the archimedean  places.  For each $v \in S$, let $L_{1,v}, \ldots, L_{m,v}$ be linearly independent linear  forms in the variables $X_1,\ldots,X_m$ with  coefficients in $K$.  For any $\varepsilon>0$, the set of solutions $\mathbf{x} \in K^m\backslash\{0\}$ to the inequality 
\begin{equation*}
\prod_{v\in S}\prod_{i=1}^{m} \frac{|L_{i,v}(\mathbf{x})|_v}{\|\mathbf{x}\|_v} \leq \frac{1}{H(\mathbf{x})^{m+\varepsilon}}   
\end{equation*}
is contained  in finitely many proper subspaces of $K^m$.
\end{theorem}

The following lemma, established in \cite{corv}, is used at several places in the proof of the main result
of \cite{corv}. 

\begin{lemma}\label{lemCZ1}
Let $K$ be a number field which is  Galois  over $\mathbb{Q}$  and $S$  be a finite set of places, containing all  the archimedean places.  Let $\sigma_1,\ldots,\sigma_n$  be distinct automorphisms of $K$ for some integer $n\geq 1$ and let 
$\lambda_1,\ldots,\lambda_n$  be non-zero elements of $K$. 
Let  $\varepsilon>0$  be a positive real number and $w\in S$  be a distinguished place.  
Let $c > 0$ be a real number and let $\mathfrak{E}\subset \mathcal{O}_S^\times$ be the set of solutions $u\in\mathcal{O}_S^\times$  
of the inequality
\begin{equation*}
0< |\lambda_1  \sigma_1(u)+\cdots+\lambda_n \sigma_n(u)|_w<
 c \max\{|\sigma_1(u)|_w,\ldots,|\sigma_n(u)|_w\} H(u)^{-\varepsilon}.
\end{equation*}
If $\mathfrak{E}$ is an infinite subset of $\mathcal{O}_S^\times$, 
then there exists a non-trivial linear relation of the form 
$$
a_1 \sigma_1(u)+\cdots+a_n \sigma_n(u)=0,\quad \mbox{with }  a_i\in K    
$$
which holds for infinitely many elements  $u$ in $\mathcal{O}_S^\times$. 
\end{lemma}

A slight modification of Lemma \ref{lemCZ1},  yields the following. 

\begin{lemma}\label{lem2.1}
Let $K$ be a number field which is  Galois  over $\mathbb{Q}$  and $S$  be a finite set of places, containing all  the archimedean places.  Let $\sigma_1,\ldots,\sigma_n$  be distinct automorphisms of $K$ for some integer $n\geq 1$ and let 
$\lambda_0, \lambda_1,\ldots,\lambda_n$  be non-zero elements of $K$. 
Let  $\varepsilon>0$  be a positive real number and $w\in S$  be a distinguished place.  Let $\mathfrak{E}\subset \mathcal{O}_S^\times\times\mathbb{Z}\backslash\{0\}$ be the subset  defined as 
\begin{equation*}\label{eq2.3}
\tag{2.2}
\mathfrak{E} := \left\{ (u,q)\in \mathcal{O}_S^\times \times\mathbb{Z}\backslash\{0\}  
 \ :  0<\ |\lambda_0+\lambda_1 q\sigma_1(u)+\cdots+\lambda_nq\sigma_n(u)|_w<\frac{\max\{|q\sigma_1(u)|_w,\ldots,|q\sigma_n(u)|_w\}}{|q|^{n+\varepsilon}H(u)^\varepsilon}\right\}.
\end{equation*}
If $\mathfrak{E}$ is infinite subset of $\mathcal{O}_S^\times\times\mathbb{Z}\backslash\{0\}$, then there exists a non-trivial linear relation of the form 
$$
a_1 \sigma_1(u)+\cdots+a_n \sigma_n(u)=0,\quad \mbox{with }  a_i\in K    
$$
which holds for infinitely many elements $u$ in $\mathcal{O}_S^\times$ along the pairs $(u,q)\in\mathfrak{E}$. 
\end{lemma}

\begin{proof}
In order to prove this lemma, we shall apply  Theorem \ref{schli} as in the proof \cite[Lemma 1]{corv}. Without loss of generality, we can assume that    
\begin{equation*}
|q\sigma_1(u)|_w=\max\{|q\sigma_1(u)|_w,\ldots,|q\sigma_n(u)|_w\}
\end{equation*}
for all $(u,q)\in \mathfrak{E}$. For $\mathit{v}\in S$, let us define $n+1$ linear forms $L_{\mathit{v},0},\ldots,L_{\mathit{v},n}$ in $n+1$ variables ${\bf  x}=(x_0,x_1,\ldots,x_n)$ as follows: Put
$ L_{w,0}(x_0,x_1,\ldots,x_n)=X_0$ and $L_{w,1}(x_0,x_1,\ldots,x_n)=\lambda_0 x_0+\lambda_1 x_1+\cdots+\lambda_n x_n$.  For $2\leq i\leq n$,  define
$L_{w,i}(x_0,x_1,\ldots,x_n)=x_i.$  Also, for each  $\mathit{v}\neq w\in S$,  and  $0\leq j\leq n $, we let
$L_{\mathit{v},j}(x_0,x_1,\ldots,x_n)=x_j.$ Take ${\bf x}=(1, q\sigma_1(u),\ldots,q\sigma_n(u))\in K^{n+1}$ 
and consider the product 
\begin{equation*}
\prod_{\mathit{v}\in S}\prod_{i=0}^n\frac{|L_{\mathit{v},i}(\bf x)|_\mathit{v}}{||\bf x||_\mathit{v}}.
\end{equation*}
Using the fact that $L_{v,j}({\bf x}) =q \sigma_j(u)$  for $2\leq j\leq n $ and  that the $\sigma_j(u)$  are $S$-units,  by the product formula,  we obtain 
\begin{equation*}\label{eq2.4}
\tag{2.3}
\prod_{\mathit{v\in S}}\prod_{j=2}^{n}|L_{\mathit{v},j}(\mathbf{x})|_\mathit{v}=\prod_{\mathit{v}\in S}\prod_{j=2}^{n}|q|_\mathit{v}\leq \prod_{v\in M_\infty}\prod_{j=2}^{n}|q|_\mathit{v}= |q|^{n-1}.
\end{equation*} 
Now we estimate $\prod_{\mathit{v}\in S}\prod_{i=0}^n||\bf x||_\mathit{v}$:
\begin{equation*}\label{eq2.5}
\tag{2.4}
\prod_{\mathit{v}\in S}\prod_{i=0}^n||{\bf x}||_\mathit{v}=\prod_{i=0}^n\left(\prod_{\mathit{v}\in S}||{\bf x}||_\mathit{v}\right)\geq (H({\bf x}))^{n+1}=H^{n+1}(1,q\sigma_1(u),\ldots,q\sigma_n(u))
\end{equation*}
since $||{\bf x}||_\mathit{v}\leq 1$ for all $\mathit{v}$ not in $S$. 
\smallskip
     
Since $L_{\mathit{v},0}({\bf x})=1$ for all $\mathit{v}\in S$ and  by the  product formula, we get    
$$
\prod_{\mathit{v}\neq w\in S}|q\sigma_1(u)|_\mathit{v}=\left(\prod_{\mathit{v}\neq w\in S}|q|\right)(|\sigma_1(u)|_w)^{-1}. 
$$

Thus  from \eqref{eq2.3}, \eqref{eq2.4} and \eqref{eq2.5}, we obtain
\begin{align*}
\prod_{\mathit{v}\in S}\prod_{i=0}^n\frac{|L_{\mathit{v},i}(\bf x)|_\mathit{v}}{||\bf x||_\mathit{v}}&\leq \frac{|\lambda_0+\lambda_1 q\sigma_1(u)+\cdots+\lambda_n q\sigma_n(u)|_w |q|^n}{|q\sigma_1(u)|_w}\frac{1}{H^{n+1}({\bf x})}\\
&\leq\frac{\max\{|q\sigma_1(u)|_w,\ldots,|q\sigma_n(u)|_w\}}{|q\sigma_1(u)|_w |(|q|H(u))^\varepsilon}\frac{1}{H({\bf x})^{n+1}}, 
\end{align*}
as $(u,q)\in\mathfrak{E}$. Using that 
$$
|q\sigma_1(u)|_w=\max\{|q\sigma_1(u)|_w,\ldots,|q\sigma_n(u)|_w\}, 
$$
we get 
$$
\prod_{\mathit{v}\in S}\prod_{i=0}^n\frac{|L_{\mathit{v},i}(\bf x)|_\mathit{v}}{||\bf x||_\mathit{v}}\leq \frac{1}{H({\bf x})^{n+1}}\frac{1}{(|q|H(u))^\varepsilon}.
$$
Since the height of the vector ${\bf x}=(1,q\sigma_1(u),\ldots,q\sigma_n(u))$ satisfies 
$
H({\bf x})\leq |q|H(u)^{[K:\mathbb{Q}]}=|q|H(u)^n,
$
the above estimate becomes 
$$
\prod_{\mathit{v}\in S}\prod_{i=0}^n\frac{|L_{i,\mathit{v}}(\bf x)|_\mathit{v}}{||\bf x||_\mathit{v}}\leq \frac{1}{H({\bf x})^{n+1}}\frac{1}{H({\bf x})^{\varepsilon/[K:\mathbb{Q}]}}=\frac{1}{H({\bf x})^{n+1+\varepsilon/[K:\mathbb{Q}]}}.
$$
Therefore  by Theorem \ref{schli}, there exists a non-trivial relation of the form
\begin{equation*}\label{eq2.6}
\tag{2.5}
a_0+a_1 q\sigma_1(u)+\cdots+a_n q\sigma_n(u)=0
\end{equation*}
satisfied by infinitely many pairs $(u,q)\in \mathfrak{E}$. In order to finish the proof, it is enough to prove the following claim.
\smallskip

\noindent{\bf CLAIM.~}  There exists a non-trivial relation as \eqref{eq2.6} with $a_0=0$. 
\bigskip

Assume that $a_0\neq 0$. By rewriting the relation \eqref{eq2.6}, we obtain 
\begin{equation*}\label{eq2.7}
\tag{2.6}
a_0=-a_1 q\sigma_1(u)-\cdots-a_n q\sigma_n(u)\iff 1=-\frac{a_1}{a_0}q\sigma_1(u)-\cdots--\frac{a_n}{a_0}q\sigma_n(u).
\end{equation*}
Since $\frac{a_1}{a_0},\ldots,\frac{a_n}{a_0}$ are not all zero, let  $\frac{a_{i_1}}{a_0},\ldots,\frac{a_{i_r}}{a_0}$ be  the non-zero elements among them. We enlarge our set $S$, so that $\frac{a_{i_1}}{a_0},\ldots\frac{a_{i_r}}{a_0}\in \mathcal{O}^\times_S$.  Since $\sigma_i(u)\in\mathcal{O}^\times_S$ for $i=1,\ldots,n$, from the relation \eqref{eq2.7}, we conclude that $q$ must be an $S$-unit.

Hence, by applying  the  $S$-unit equation theorem of Evertse and van der Poorten-Schlickewei  \cite[Theorem II.4]{zannier} \cite[Theorem II.4]{zannier}(see also \cite{bomb}, \cite{schmidt}) to the relation \eqref{eq2.7},  there exists a non-trivial relation of the form 
$$
a_{i_1}\sigma_{i_1}(u)+\cdots+a_{i_s}\sigma_{i_s}(u)=0 
$$
which holds for infinitely many values of $u$  coming from the pairs $(u,q)\in\mathfrak{E}$.   
This proves the claim and hence the lemma.
\end{proof}

\section{A Key lemma for the proof of Theorem \ref{maintheorem}}

The following lemma is key to  the proof of Theorem \ref{maintheorem}  and its proof is based on the Subspace Theorem along with the idea in \cite{corv}, with  various modifications.

\begin{lemma}\label{lem3.1}
Let $K$ be a Galois extension over $\mathbb{Q}$ of degree $n$ and $k\subset K$  be a subfield of degree $d$  over $\mathbb{Q}$. Let $\delta, \beta$ be two non-zero elements of $K$ with $\beta$  irrational. 
Let $S$  be a finite set of places on $K$ containing all the archimedean places 
and  let $\varepsilon>0$  be a given real number.   Let 
\begin{equation*}\label{eq3.1}
\tag{3.1}
\mathcal{B} = \left\{(u, q, p)\in (\mathcal{O}_S^\times\cap k)\times\mathbb{Z}^2 \ : \   
0<|\delta qu+\beta-p|<\frac{1}{H(u)^\varepsilon q^{d+\varepsilon}}\right\}
\end{equation*}
such that for each triple $(u,q,p)\in\mathcal{B}$,  $|\delta q u|>1$. If $\mathcal{B}$ is infinite, then there exist a proper subfield $k'\subset k$, a non-zero element $\delta'$ in $k$  and an infinite subset $\mathcal{B}'\subset \mathcal{B}$  such that for all triples $(u, q, p)\in\mathcal{B}'$ 
we have $u/\delta'\in k'.$
\end{lemma}

\begin{proof}
Since $\mathcal{B}$ is an infinite set of solutions of \eqref{eq3.1},  we first observe that 
we may assume that $H(u)\rightarrow\infty$.  

Suppose that $H(u)$ is bounded. Then there exists an infinite subset $\mathcal{A}$ of $\mathcal{B}$ such that the number $u$ is constant for all elements in  $\mathcal{A}$, say $u_0$ for all triples $(u,q,p)\in\mathcal{A}$ and $q$ is unbounded along the set $\mathcal{A}$. Now we apply Theorem \ref{schli} 
to the field $\mathbb{Q}$ 
with the input $S=\{\infty\}$, 
linear forms 
$L_{1,\infty}(x_1, x_2,x_3)=\delta u_0 x_1+\beta x_2-x_3$, $L_{i,\infty}(x_1,x_2,x_3)=x_i$ 
for $2\leq i\leq 3$  
and the points   $(q,1,p)$.  From \eqref{eq3.1}, we see that there is  a $\eta > 0$ such that the inequality 
$$
\prod_{i=1}^3|L_{i,\infty}(q,1,p)|_\infty\leq \frac{1}{(\max\{|q|,1,|p|\})^{\eta}}
$$
holds for infinitely many triples $(q,1,p)\in\mathbb{Z}^3$.  Thus by Theorem \ref{schli}, there exists a proper subspace of $\mathbb{Q}^3$ containing infinitely many triples $(q,1,p)$, i.e.,  we have a non-trivial relation of the form
$$
a_0+a_1p+a_2q=0
$$
satisfied by infinitely many triples of the form $(q,1,p)$. Since $a_i$'s are  integers   
and $q\rightarrow\infty$ along the set $\mathcal{A}$, we conclude that $a_1\neq 0$.  By substituting the value of $p$ into the inequality \eqref{eq3.1} along the set $\mathcal{A}$, we get
$$
0<\left|\delta q u_0+\beta+\left(\frac{a_0}{a_1}+\frac{a_2}{a_1}q\right)\right|\leq \frac{1}{H(u_0)^\varepsilon}\frac{1}{q^{d+\varepsilon}} \ \iff \  0<\left|\left(\delta u_0+\frac{a_2}{a_1}\right)q+\beta+\frac{a_0}{a_1}\right|\leq \frac{1}{H(u_0)^\varepsilon}\frac{1}{q^{d+\varepsilon}}, $$ 
which is not true as $q\to\infty$.  Therefore, we conclude that $H(u)\to\infty$  along the set $\mathcal{B}$.
\bigskip

Let $\mathcal{H}:=\mbox{Gal}(K/k)\subset\mbox{Gal}(K/\mathbb{Q})=\mathcal{G}$  be the subgroup of the Galois group $\mathcal{G}$ fixing $k$.  
Since $K$ is Galois over $\mathbb{Q}$, we have $K$ is Galois over $k$ and  $|\mathcal{G}/\mathcal{H}| =d$. Therefore, among the $n$ embeddings of $K$, there are exactly $d$ embeddings $\sigma_1,\ldots,\sigma_d$, which  are  the representatives for the left cosets of $\mathcal{H}$ in $\mathcal{G}$ with $\sigma_1$ being the identity. More precisely, 
$$
\mathcal{G}/\mathcal{H}:=\{\mathcal{H}, \sigma_2 \mathcal{H},\ldots,\sigma_d \mathcal{H}\}.
$$
Each automorphism $\rho\in\mbox{Gal}(K/\mathbb{Q})$ defines an archimedean absolute value  on $K$ by the formula
\begin{equation*}\label{eq3.2}
\tag{3.2}
|x|_\rho=|\rho^{-1}(x)|^{d(K)/[K:\mathbb{Q}]},
\end{equation*}
where $| \cdot |$  denotes the usual complex absolute value and $d(K)=1$ if $K\subset \mathbb{R}$ and $d(K)=2$ otherwise. Let $\rho_1$ and $\rho_2$ be two distinct automorphism on $K$, which give rise to the same archimedean absolute values $\mathit{v}$  if and only if 
$\rho^{-1}_1\circ \rho_2$ is a complex conjugation. 
Then   for each $\rho\in\mbox{Gal}(K/\mathbb{Q})$, by \eqref{eq3.2},  we have 
\begin{equation*}\label{eq3.3}
\tag{3.3}
|\delta qu+\beta-p|^{d(K)/[K:\mathbb{Q}]}=|\rho(\delta)\rho(qu)+\rho(\beta)-\rho(p)|_\rho=|\rho(\delta)q\rho(u)+\rho(\beta)-p|_\rho.
\end{equation*}
For each $\mathit{v}\in M_\infty$, let $\rho_\mathit{v}$ be an automorphism defining the valuation $\mathit{v}$ according to \eqref{eq3.2}: $|\alpha|_\mathit{v}:=|\alpha|_{\rho_\mathit{v}}$. Then the set $\{\rho_\mathit{v}:\mathit{v}\in M_\infty\}$ denotes the left cosets of the subgroup generated by the complex conjugation in $\mathcal{G}$.
\smallskip

Denote by $i: K\to \mathbb{C}$, the embedding given by $\alpha \mapsto \bar{\alpha}$, the complex conjugation. Then for each $j = 1, \ldots, d,$ let
$$
S_j = \left\{v\in M_\infty\ : \ \rho_v\vert_k = i\circ\sigma_j: k\hookrightarrow \mathbb{C}\right\}
$$
and hence $S_1\cup\ldots\cup S_d=M_\infty$.    We keep this notation throughout the paper. Now we take the product of the terms in \eqref{eq3.3} where $\rho$ runs through the set $\{\rho_\mathit{v}:\mathit{v}\in M_\infty\}$ to obtain 
\begin{equation*}\label{eq3.4}
\tag{3.4}
\prod_{\mathit{v}\in M_\infty}|\rho_\mathit{v}(\delta)\rho_{\mathit{v}}(qu)+\rho_\mathit{v}(\beta)-p|_\mathit{v}=\prod_{j=1}^d\prod_{\mathit{v}\in S_j}|\rho_{\mathit{v}}(\delta)\sigma_j(qu)+\rho_{\mathit{v}}(\beta)-p|_v.
\end{equation*} 
By \eqref{eq3.3}, we see that 
$$
\prod_{\mathit{v}\in M_\infty}|\rho_\mathit{v}(\delta)\rho_{\mathit{v}}(qu)+\rho_\mathit{v}(\beta)-p|_\mathit{v}=\prod_{\mathit{v}\in M_\infty}|\delta q u+\beta-p|^{d(K)/[K:\mathbb{Q}]}=|\delta q u+\beta-p|^{{\sum_{\mathit{v}\in M_\infty}}d(K)/[K:\mathbb{Q}]}.
$$
From \eqref{eq3.4} and the formula $\displaystyle\sum_{\mathit{v}\in M_\infty}d(K)=[K:\mathbb{Q}]$, it follows that 
\begin{equation*}\label{eq3.5}
\tag{3.5}
\prod_{j=1}^d\prod_{\mathit{v}\in S_j}|\rho_{\mathit{v}}(\delta)\sigma_j(qu)+\rho_{\mathit{v}}(\beta)-p|_v=|\delta q u+\beta-p|.
\end{equation*}  

Now, for each $\mathit{v}\in S$, we define $d+2$ linearly independent linear forms in $d+2$  variables as follows: For $j = 1, 2, \ldots, d$ and for  $\mathit{v}\in S_j$,  let  
\begin{eqnarray*}
L_{\mathit{v},0}(x_0,x_1,\ldots,x_{d+1})&=& \rho_{\mathit{v}}(\beta) x_0 - x_1 + \rho_{\mathit{v}}(\delta) x_{j+1}\\   
L_{\mathit{v},1}(x_0,x_1,\ldots,x_{d+1})&=& x_0,
\end{eqnarray*}
 and for $2\leq i\leq d+1$, put
$$
L_{v, i}(x_1,\ldots,x_{d+1})=x_i.
$$
Also, for  $\mathit{v}\in S\backslash{M_\infty}$ and for  $0\leq i \leq d+1$,   let  
$$
L_{v,i}(x_1,\ldots,x_{d+1})=x_i.
$$
Take points  $\mathbf{x}$ in $K^{d+2}$ as
$$
\mathbf{x}=(1,p,q\sigma_1(u),\ldots,q\sigma_d(u)) \in K^{d+2}.
$$
In order to apply  Theorem \ref{schli}, we need to calculate the following quantity 
\begin{equation*}\label{eq3.6}
\tag{3.6}
\prod_{\mathit{v\in S}}\prod_{i=0}^{d+1}\frac{|L_{\mathit{v},i}(\mathbf{x})|_\mathit{v}}{||\mathbf{x}||_\mathit{v}}.
\end{equation*}
Using the fact that $L_{\mathit{v},i}(\mathbf{x})=q\sigma_i(u)$, for $2\leq i\leq d+1$ and that the $\sigma_j(u)$'s are $S$-units, by the product formula, we obtain

\begin{equation*}\label{eq3.7}
\tag{3.7}
\prod_{\mathit{v\in S}}\prod_{i=2}^{d+1}|L_{\mathit{v},i}(\mathbf{x})|_\mathit{v}=\prod_{\mathit{v}\in S}\prod_{i=2}^{d+1}|q|_\mathit{v}\leq \prod_{v\in M_\infty}\prod_{i=2}^{d+1}|q|_\mathit{v}=\prod_{i=2}^{d+1}|q|^{\sum_{\mathit{v}\in M_\infty}d(K)/[K:\mathbb{Q}]}\leq |q|^d.
\end{equation*}
%
Since $||\mathbf{X}||_\mathit{v}\leq  1$  for all $\mathit{v}\notin S$, we estimate the denominators in \eqref{eq3.6} as 
\begin{equation*}\label{eq3.8}
\tag{3.8}
\prod_{\mathit{v\in S}}\prod_{i=0}^{d+1}||\mathbf{x}||_\mathit{v}\geq \prod_{\mathit{v\in M_K}}\prod_{i=0}^{d+1}||\mathbf{x}||_\mathit{v}=\prod_{i=0}^{d+1}\left(\prod_{\mathit{v\in M_K}}||\mathbf{x}||_\mathit{v}\right) = \prod_{i=0}^{d+1}H(\mathbf{x})\geq H(\mathbf{x})^{d+2}.
\end{equation*}
By \eqref{eq3.6}, \eqref{eq3.7}  and \eqref{eq3.8}, it follows that 
$$
\prod_{\mathit{v\in S}}\prod_{i=0}^{d+1}\frac{|L_{\mathit{v},i}(\mathbf{x})|_\mathit{v}}{||\mathbf{x}||_\mathit{v}}\leq \frac{1}{H(\mathbf{x})^{d+2}}|q|^d|\delta q u+\beta-p|.
$$
Thus, from \eqref{eq3.1}, we have
$$
\prod_{\mathit{v\in S}}\prod_{i=0}^{d+1}\frac{|L_{\mathit{v},i}(\mathbf{x})|_\mathit{v}}{||\mathbf{x}||_\mathit{v}}\leq \frac{1}{H(\mathbf{x})^{d+2}}|q|^d\frac{1}{H^{\varepsilon}(u)}\frac{1}{|q|^{d+\varepsilon}}=\frac{1}{H(\mathbf{x})^{d+2}}\frac{1}{(|q|H(u))^{\varepsilon}}.
$$
Notice that 
\begin{eqnarray*}
H(\mathbf{x})&=&\prod_{\mathit{v}\in M_K}\mbox{max}\{1, |p|_\mathit{v},|q\sigma_1(u)|_\mathit{v},\ldots,|q\sigma_d(u)|_\mathit{v}\}
\leq\prod_{\mathit{v}\in S}\mbox{max}\{1, |p|_\mathit{v},|q\sigma_1(u)|_\mathit{v},\ldots,|q\sigma_d(u)|_\mathit{v}\}\\
&\leq& \prod_{\mathit{v}\in S}\mbox{max}\{1,|p|_\mathit{v},|q|_\mathit{v} \}\prod_{\mathit{v}\in S}\mbox{max}\{1,|\sigma_1(u)|_\mathit{v},\ldots,|\sigma_d(u)|_\mathit{v}\}\\
&\leq& \max\{|p|, |q|\}\left(\prod_{\mathit{v}\in S}\mbox{max}\{1,|\sigma_1(u)|_\mathit{v}\}\right)\cdots \left(\prod_{\mathit{v}\in S}\mbox{max}\{1,|\sigma_d(u)|_\mathit{v}\}\right)=\max\{|p|, |q|\}H(u)^d.
\end{eqnarray*} 
By using the inequality 
$
||x|-|y||\leq |x-y|
$
and since  $H(u)\to \infty$ for $(u, q, p)\in \mathcal{B}$, from \eqref{eq3.1}, we conclude that $|p|\leq |\delta q u+\beta|+1.$  Since $|u|^{\frac{1}{d}}\leq H(u)$,   we get that
$$
|p|\leq |\delta q u+\beta|+1\leq |q||\delta+\beta|H^d(u)+1\leq |q|H^{2d}(u)
$$
for all but finitely many triples $(u, q, p)\in \mathcal{B}$. By combining both these  inequalities, we obtain $H(\mathbf{x})\leq |q| H(u)^{3d}$, and hence   $H(\mathbf{x})^{1/3d} \leq |q|H(u)$. Therefore, we get 
$$
\prod_{\mathit{v\in S}}\prod_{i=0}^{d+1}\frac{|L_{\mathit{v},i}(\mathbf{x})|_\mathit{v}}{||\mathbf{x}||_\mathit{v}}\leq \frac{1}{H(\mathbf{x})^{d+2}}\frac{1}{(|q|H(u))^\varepsilon}\leq \frac{1}{H(\mathbf{x})^{d+2+(\varepsilon)/3d}} = \frac{1}{H(\mathbf{x})^{d+2+ \varepsilon'}},
$$
for some $\varepsilon' >0$ and for infinitely many tuples $(1,p,q\sigma_1(u),\ldots,q\sigma_d(u))$ along the triples $(u,q,p)\in\mathcal{B}$.  By  Theorem \ref{schli}, there exists a proper subspace of $K^{d+2}$ containing infinitely many $\mathbf{x}=(1,p,q\sigma_1(u),\ldots,q\sigma_d(u))$ along the triples  $(u,q,p) \in \mathcal{B}$, i.e.,    
 we have a non-trivial linear relation of the form
\begin{equation*}\label{eq3.9}
\tag{3.9}
a_0+a_1 p+b_1q\sigma_1(u)+\cdots+b_d q \sigma_d(u)=0,\quad a_i, b_j\in K,
\end{equation*}
satisfied by all the triples $(u, q, p)\in\mathcal{B}_1$ for an infinite subset of $\mathcal{B}_1\subset\mathcal{B}$.   
\bigskip

Under the hypotheses of the Main Theorem in \cite{corv}, the authors established 
the existence of such a non-trivial linear relation with $a_0 = 0$. 
The present situation is slightly more complicated. 
As in \cite{corv}, we will establish that there is 
a non-trivial linear relation as above with $a_0 = a_1 = 0$, and then we will conclude exactly as in \cite{corv}.
\smallskip

\noindent{\bf Claim 1.}  At least one of the $b_j$'s is non-zero in the relation \eqref{eq3.9}.
\bigskip

If not, suppose $b_i=0$  for all $1\leq i\leq d$.  Then from \eqref{eq3.8}, we have 
\begin{equation*}\label{eq3.10}
\tag{3.10}
0\neq p=\frac{-a_0}{a_1}\in K.
\end{equation*}
We deduce from \eqref{eq3.1} and \eqref{eq3.10} that
\begin{equation*}\label{eq3.11}
\tag{3.11}
0<\left|\delta q u+\beta+\frac{a_0}{a_1}\right|<\frac{1}{H(u)^\varepsilon q^{d+\varepsilon}}
\end{equation*}
holds for infinitely many pairs $(u,q)$ along the set $\mathcal{B}_1$.  Since $\beta$ is an irrational, from \eqref{eq3.10} we have $\beta+\frac{a_0}{a_1}\neq 0$.  We then apply Theorem \ref{schli} with 
$S$ being the finite set composed of the archimedean places on $K$, the linear 
forms $L_{\mathit{v},1}(x_1, x_2)=(\beta+\frac{a_0}{a_1})x_1+\delta x_2$, $L_{\mathit{v},2}(x_1, x_2)=x_1$~~ for $\mathit{v}\in S$, 
and the pairs $(1,qu)\in K^2$. Thus  by Theorem \ref{schli}, we get a  non-trivial relation of the form 
$$
c_0+c_1 qu=0,
$$
which holds for infinitely many pairs $(u,q)$ along the set $\mathcal{B}_1$.  This implies that $qu$  is a constant for infinitely many pairs $(u,q)\in\mathcal{B}_1$. However, this violates the inequality \eqref{eq3.11} because $H(u)\to\infty$ as we vary $(u,q)$ in $\mathcal{B}_1$. Therefore we conclude that at  least one of the $b_j$'s is non-zero in the relation \eqref{eq3.9}.
\bigskip

\noindent{\bf Claim 2.}  There exists a non-trivial relation as \eqref{eq3.9} with $a_0=a_1=0$.
\smallskip

Suppose  that $a_0\neq 0$. Then by re-writing the  relation \eqref{eq3.9}, we obtain
\begin{equation*}\label{eq3.12}
\tag{3.12}
\beta=-\beta\left(\frac{a_1}{a_0}p+\frac{b_1}{a_0}q\sigma_1(u)+\cdots+\frac{b_d}{a_0}q\sigma_d(u)\right).
\end{equation*}
Substituting the value of $\beta$ from \eqref{eq3.12} in \eqref{eq3.1}, we get

\begin{equation*}\label{eq3.13}
\tag{3.13}
0<\left|\delta q u-(\beta a_1/a_0+1)p-\beta\left(\frac{b_1}{a_0}q\sigma_1(u)+\cdots+\frac{b_d}{a_0}q\sigma_d(u)\right)\right|<\frac{1}{H(u)^\varepsilon}\frac{1}{q^{d+\varepsilon}}.
\end{equation*}
\vspace{.1cm}

The rest of the proof of this claim divided into   two cases, according to $\beta a_1/a_0+1$ is $0$ or not. 
\bigskip

\noindent{\bf Case 1.} $\displaystyle\beta \frac{a_1}{a_0}+1 =0$.

In this case, the relation \eqref{eq3.12}  can be written as 
\begin{equation*}\label{eq3.14}
\tag{3.14}
\beta-p=-\left(\frac{b_1}{a_1}q \sigma_1(u)+\cdots+\frac{b_d}{a_1}q\sigma_d(u)\right).
\end{equation*}
Since $K$ over $\Q$ is Galois and $\beta$  is an algebraic irrational,  there exists an automorphism $\rho_0\in\mbox{Gal}(K/\mathbb{Q})$ such that $\rho_0(\beta)\neq \beta$.  By applying the automorphism $\rho_0$ on both sides of the equality \eqref{eq3.14}, we get 
\begin{eqnarray*}
\rho_0(\beta)-p&=&-\left[\rho_0\left(\frac{b_1}{a_1}\right)q \rho_0\circ\sigma_1(u)+\cdots+\rho_0\left(\frac{b_d}{a_1}\right)q\rho_0\circ\sigma_d(u)\right] \\
&=& -\left[\rho_0\left(\frac{b_1}{a_1}\right)q \sigma_{1,0}(u)+\cdots+\rho_0\left(\frac{b_d}{a_1}\right)q\sigma_{d,0}(u)\right],
\end{eqnarray*}
as the restriction of $\rho_0$ on $k$ belongs to $\{\sigma_1,\ldots,\sigma_d\}$ and  hence    $\sigma_{i,0}\in\{\sigma_1,\ldots,\sigma_d\}$ for $1\leq i\leq d$. Now by subtracting this equality from \eqref{eq3.14}, we obtain  
$$
0\neq \rho_0(\beta)-\beta=c_1 q \sigma_1(u)+\cdots+c_d q\sigma_d(u) := \gamma,\quad c_i\in K
$$
for all the pairs $(q,u)$  along the  triples $(u,q,p)\in\mathcal{B}_1$ with    $\gamma=\rho_0(\beta)-\beta.$ 
We can easily see that in this   relation  at least one of $c_i$'s is non-zero, say $c_{i_1},\ldots, c_{i_s}$ are non-zero elements among them, where $\{i_1,\ldots,i_s\}\subset\{1,\ldots,d\}$.  
Dividing this equality by $\gamma$, we get the non-trivial relation of the kind 
\begin{equation*}\label{eq3.16}
\tag{3.15}
1=\left(\frac{c_{i_1}}{\gamma}q \sigma_{i_1}(u)+\cdots+\frac{c_{i_s}}{\gamma}q \sigma_{i_s}(u)\right).
\end{equation*}
As we have seen in the proof of Lemma \ref{lem2.1}, we enlarge our set $S$ so that $\frac{c_{i_1}}{\gamma},\ldots,\frac{c_{i_s}}{\gamma}\in \mathcal{O}^\times_S$.  Thus  from the relation \eqref{eq3.16}, we also conclude that $q\in\mathcal{O}^\times_S$. We can apply  the  $S$-unit equation theorem of Evertse and van der Poorten-Schlickewei  \cite[Theorem II.4]{zannier} to the relation \eqref{eq3.16}, which entails that there exists a non-trivial relation of the form
\begin{equation*}
b_{i_1} \sigma_{i_1}(u)+\cdots+b_{i_s}\sigma_{i_s}(u)=0
\end{equation*}
holds for infinitely many $u$  coming from the  triples $(u, q,p)\in \mathcal{B}_1$  for an  infinite subset $\mathcal{B}_1\subset\mathcal{B}$.
\bigskip

\noindent{\bf Case 2.~}  $\beta a_1/a_0+1\neq 0$.
\vspace{.2cm}

By  \eqref{eq3.13}, we have
\begin{equation*}\label{eq3.24}
\tag{3.16}
0<\left|\delta q u-\left(\frac{\beta a_1}{a_0}+1\right)p-\beta\left(\frac{b_1}{a_0}q\sigma_1(u)+\cdots+\frac{b_d}{a_0}q\sigma_d(u)\right)\right|<\frac{1}{H(u)^\varepsilon}\frac{1}{q^{d+\varepsilon}}.
\end{equation*}
 We follow    the similar procedure to the inequality \eqref{eq3.24}   as we have seen in the beginning of this lemma  to  get the following
\begin{align*}
&\prod_{j=1}^d\prod_{\mathit{v}\in S_j}\left|\rho_\mathit{v}\left(\delta-\beta\frac{b_1}{a_0}\right)q\sigma_{\mathit{v}(1)}(u)-\rho_\mathit{v}\left(\frac{\beta a_1}{a_0}+1\right)p-\rho_\mathit{v}\left(\frac{\beta b_2}{a_0}\right)q\sigma_{\mathit{v}(2)}(u)-\cdots-\rho_\mathit{v}\left(\frac{\beta b_d}{a_0}\right) q\sigma_{\mathit{v}(d)}(u))\right|_\mathit{v}\\&=\left|\left(\delta-\beta \frac{b_1}{a_0}\right) q \sigma_1(u) -(\beta a_1/a_0+1)p-\beta\left(\frac{b_2}{a_0}q\sigma_2(u)+\cdots+\frac{b_d}{a_0}q\sigma_d(u)\right)\right|<\frac{1}{H(u)^\varepsilon}\frac{1}{q^{d+\varepsilon}},
\end{align*}
where for  each $\mathit{v}\in M_\infty$ and $j=1,2,\ldots,d$,  we have set $\rho_\mathit{v}\circ\sigma_j=\sigma_{\mathit{v}(j)}$  on the field $k$  and $\{\mathit{v}(1), \ldots,\mathit{v}(d)\}$ is a permutation of $\{1,\ldots,d\}$.  Now for each $\mathit{v}\in S$, we define $d+1$ linearly independent linear forms in $d+1$ variables as follows: for $j=1,\ldots,d$ and for each $\mathit{v}\in S_j$ define
\begin{align*}
L_{\mathit{v},0}(x_0, x_1,\ldots,x_d)&=-\rho_{\mathit{v}}\left(\frac{\beta a_1}{a_0}+1\right)x_0+\rho_\mathit{v}\left(\delta-\beta\frac{b_1}{a_0}\right)x_{\mathit{v}(1)}-\rho_{\mathit{v}}\left(\frac{\beta b_2}{a_0}\right)x_{\mathit{v}(2)}-\cdots-\rho_{\mathit{v}}\left(\frac{\beta b_d}{a_0}\right)x_{\mathit{v}(d)}
\end{align*}
and for $1\leq i\leq d+1$, define 
$
L_{v, i}(x_1,\ldots,x_{d+1})=x_i,
$
Also for $\mathit{v}\in S\backslash{M_\infty}$ and   $0\leq i \leq d+1$,   let  
$
L_{v,i}(x_1,\ldots,x_{d+1})=x_i.
$
Since  in this case   $\beta\frac{a_1}{a_0}+1$ is non-zero,  we see that the linear forms $L_{\mathit{v},0},\ldots,L_{\mathit{v},d}$ are linearly independent for each $\mathit{v}\in S$. Finally, let $\mathbf{x}$ be the point in $K^{d+1}$, which is of the form
$$
{\mathbf{x}}=(p, q\sigma_1(u),\ldots,q\sigma_d(u)).
$$
Then by using Theorem \ref{schli} similar to the first part   of this lemma, we get a non-trivial relation of the form
$$
a_1 p+b_1 q \sigma_1(u)+\cdots+b_d q \sigma_d(u)=0.
$$
Now we prove that there exists a relation with $a_1=0$. In order to prove this, we follow the similar method as in \cite[Lemma 3, Claim]{corv} together with   Lemma \ref{lem2.1}.  If $a_1\neq 0$, then we have 
\begin{equation*}\label{eq3.17}
\tag{3.17}
p=-\frac{b_1}{a_1} q \sigma_1(u)-\cdots-\frac{b_d}{a_1} q\sigma_d(u).
\end{equation*}
First suppose that   $\displaystyle\sigma_j\left(\frac{b_1}{a_1}\right) \ne \frac{b_j}{a_1}$ for some $j$ with $2\leq j\leq d$. By applying the automorphism $\sigma_j$  on both sides of \eqref{eq3.17} and subtracting it from \eqref{eq3.17}, we obtain a non-trivial   relation of the form
\begin{equation*}
b_1 q \sigma_1(u)+\cdots+b_d q \sigma_d(u)=0,\quad  \mbox{with~~} b_i\in K.
\end{equation*}

We now assume that   $\displaystyle\frac{b_j}{a_1}= \sigma_j\left(\frac{b_1}{a_1}\right)$  for all $2\leq j\leq d$.  
\smallskip

Note that $b_1 \neq 0$. If not, then $0 =\sigma_j(b_1/a_1) = b_j/a_1$ for every $j$.  Hence  $b_i=0$ for all $i$, which contradicts Claim 1. Therefore, we can assume that $b_1\neq 0$.   By putting $\lambda = -b_1/a_1$,   we re-write \eqref{eq3.17} as 
\begin{equation*}\label{eq3.18}
\tag{3.18}
p=-q(\sigma_1(\lambda)\sigma_1(u)+\cdots+\sigma_d(\lambda)\sigma_d(u)).
\end{equation*}
 Since $b_j \in K$, it may happen that $\lambda$ does not belong to $k$. If $\lambda\notin k$,  then there exists an automorphism $\tau\in \mathcal{H}$ with $\tau(\lambda)\neq \lambda$.  By applying the automorphism $\tau$ on both sides of \eqref{eq3.18} to eliminate $p$, we obtain the linear relation 
$$
(\lambda-\tau(\lambda))q\sigma_1(u)+q\sum_{i=2}^d(\sigma_i(\lambda)\sigma_i(u)-\tau \circ\sigma_i(u))=0.
$$
Note that $\tau \circ\sigma_j$  coincides on $k$ with some $\sigma_i$ and  since $\tau\in \mathcal{H}$  and $\sigma_2,\ldots,\sigma_d \not\in \mathcal{H}$, none of the $\tau \circ \sigma_j$ with $j\geq 2$ belongs in $\mathcal{H}$.  Hence the above relation can be viewed   as a linear combination of   $\sigma_i(u)$'s with the property that the coefficient of $\sigma_1(u)$ will remain $\lambda-\tau(\lambda)$ and which is no-zero.  Therefore, we obtain the required  non-trivial relation among  $\sigma_i(u)$ as desired. 
\smallskip


Hence,  we can  assume that $\lambda\in k$ and substitute value of $p$ from \eqref{eq3.18} into \eqref{eq3.1}, we get that 

\begin{equation*}\label{eq3.20}
\tag{3.19}
0<\left|-\beta+(\lambda-\delta)q\sigma_1(u)+q\sigma_2(\lambda)\sigma_2(u)+\cdots+q\sigma_d(\lambda)\sigma_d(u)\right|<\frac{1}{H^\varepsilon(u)q^{d+\varepsilon}}
\end{equation*}
holds for infinitely many pairs $(u,q)$ along the triples $(u,q,p)\in\mathcal{B}$.
\smallskip

 If $\lambda=\delta$, then by \eqref{eq3.20}, we have
\begin{equation*}\label{eq3.21}
\tag{3.20}
0<\left|-\beta+q\sigma_2(\lambda)\sigma_2(u)+\cdots+q\sigma_d(\lambda)\sigma_d(u)\right|<\frac{1}{H^\varepsilon(u)q^{d+\varepsilon}}.
\end{equation*}
If  $\displaystyle\max\{|\sigma_2(q\lambda u)|,\ldots,|\sigma_d(q\lambda u)|\}< \frac{|\beta|}{2d}$
for all pairs $(q,u)$ satisfying \eqref{eq3.21}, then, we get 
$$
\left|-\beta+q\sigma_2(\lambda)\sigma_2(u)+\cdots+q\sigma_d(\lambda)\sigma_d(u)\right|\geq \frac{|\beta|}{2}.
$$
Therefore by \eqref{eq3.21}, we have
\begin{equation*}\label{eq3.23}
\tag{3.21}
\frac{|\beta|}{2}\leq \left|-\beta+q\sigma_2(\lambda)\sigma_2(u)+\cdots+q\sigma_d(\lambda)\sigma_d(u)\right|<\frac{1}{H^\varepsilon(u)q^{d+\varepsilon}}.
\end{equation*}
Since $H(u)\to\infty$ along   infinitely many pairs $(u,q)$ satisfy \eqref{eq3.21} and $\beta$ is non-zero, we see that  the inequality \eqref{eq3.23} can have only finitely many solutions in $(q,u)$, a contradiction. Therefore we must have  
$$
\max\{|\sigma_2(q\lambda u)|,\ldots,|\sigma_d(q\lambda u)|\}\geq \frac{|\beta|}{2d}
$$
holds for all but finitely many   pairs $(q,u)$ satisfying \eqref{eq3.21}. 
Thus from \eqref{eq3.21}, we conclude that 
$$
0<\left|-\beta+q\sigma_2(\lambda)\sigma_2(u)+\cdots+q\sigma_d(\lambda)\sigma_d(u)\right|<\frac{1}{H^\varepsilon(u)q^{d+\varepsilon}}<\frac{C\max\{|q\sigma_2(u)|,\ldots,|q\sigma_d(u)|\} }{H^\varepsilon(u)q^{d+\varepsilon}},
$$
where $C=\frac{2d(\max\{|\sigma_2(\lambda)|,\ldots,|\sigma_d(\lambda)|\})}{\min\{1,|\beta|\}}$. Hence by Lemma \ref{lem2.1}, we get a non-trivial relation as desired.  
\smallskip

Now we assume that $\lambda\neq \delta$. In this case the term $(\lambda-\delta)q\sigma_1(u)$ does appear in \eqref{eq3.20}. By applying Lemma \ref{lem2.1} with the distinguished place $w$ as in the case $\lambda=\delta$ and with the inputs $n=d$, $\lambda_1=(\lambda-\delta)$ and $\lambda_i=\sigma_i(\lambda)$ for $i=2,\ldots,d$ we  conclude the same as in the case $\delta=\lambda$. 
\bigskip

Thus by combining all the cases, we  obtain  a non-trivial relation of the form 
$$
b_1 \sigma_1(u)+\cdots+b_d  \sigma_d(u)=0, \quad b_i\in K
$$
 for infinitely many   $u$ along the triples $(u,q,p)\in\mathcal{B}$. This proves our Claim 2.
We then conclude the proof of the theorem exactly as in \cite[Lemma 3]{corv}. 
\end{proof}

\section{Proofs} 
\noindent{\bf Proof of Theorem \ref{maintheorem}.~} Since $\Gamma$ is a finitely generated multiplicative subgroup of $\overline{\mathbb{Q}}^\times$, by enlarging $\Gamma$ if necessary, we can reduce to the situation where  $\Gamma\subset\overline{\mathbb{Q}}^\times$ is the group of $S$-units, namely,
$$
\Gamma=\mathcal{O}^\times_S=\{u\in K:\prod_{\mathit{v}\in S}|u|_\mathit{v}=1\}
$$
of a suitable Galois extension $K$  over $\mathbb{Q}$ containing  $\delta,\beta$  and for a suitable finite set $S$ of places of $K$ containing all the archimedean places. Also, $S$ is stable under Galois conjugation. 
\smallskip

Suppose that the conclusion of Theorem \ref{maintheorem} is not true. Then  there exists an infinite subset $\mathcal{B} \subset\Gamma\times \mathbb{Z}^2$ of solutions $(u,q,p)$   to the inequality  \eqref{eq1.1}.  Then  inductively,  we construct  a sequence $\{\delta_i\}_{i=0}^\infty$  of elements of  $K$,  an infinite decreasing chain $\mathcal{B}_i$ of an infinite subset of $\mathcal{B}$  and an infinite strictly decreasing chain $k_i$ of subfields of $K$ with the following properties:
\bigskip

{\it For each integer $n\geq 0$, $\mathcal{B}_n\subset (k_n\times\mathbb{Z}^2)\cap\mathcal{B}_{n-1}$, $k_n\subset k_{n-1}$, $k_n\neq k_{n-1}$  and for all but finitely many triples $(u, q,p)\in\mathcal{B}_n$   satisfying  
\begin{equation*}\label{eq4.1}
\tag{4.1}
|\delta_0\cdots\delta_n q u+\beta-p|<\frac{1}{H(u)^{\varepsilon/(n+1)}q^{d+\varepsilon}}.
\end{equation*}}
If such a sequence exists, then  we  eventually get a contradiction to  the fact that the number field $K$ does not admit an infinite  strictly  decreasing chain of subfields.  Thus  in order to complete the proof of the theorem, it is enough to construct such a sequence.
\smallskip

We proceed our construction by applying induction on $n$: for $n=0$, put $\delta_0=\delta$,  $k_0=K$  and $\mathcal{B}_0=\mathcal{B}$, and we are done  in this case, since by our supposition the inequality 
$$
|\delta_0 q u+\beta-p|<\frac{1}{H(u)^\varepsilon q^{d+\eps}}
$$ 
has infinitely many solutions in triples $(u, q, p)$. 
Then by the induction hypothesis, we assume that  $\delta_n$,  $k_n$  and $\mathcal{B}_n$  exist for an integer $n\geq 0$ such that  \eqref{eq4.1} holds.
Then by  Lemma \ref{lem3.1} to the choices $\delta=\delta_0\delta_1\cdots\delta_n$ and $k=k_n$, we obtain  an element $\delta_{n+1}\in k_n$, a proper subfield $k_{n+1}$ of $k_n$  and an infinite set $\mathcal{B}_{n+1}\subset\mathcal{B}_n$  such that all triples $(u, q, p)\in\mathcal{B}_{n+1}$  satisfy $u=\delta_{n+1}u'$  with $u'\in k_{n+1}$.
Since $u'\in K$, $H(\delta_{n+1}u')\geq H(\delta)^{-1} H(u')$, we have in  particular that for almost for all $u'\in K$, $H(\delta_{n+1} u')\geq H^{\frac{n+1}{n+2}}(u')$.  Therefore by replacing $u$ by $\delta_{n+1}u'$, for all but finitely many triples $(u',q,p)\in\mathcal{B}_{n+1}$, we have the following inequality
$$
|\delta_0\delta_1\cdots\delta_n q \delta_{n+1}u'+\beta-p|<\frac{1}{H(u')^{\varepsilon/(n+2)}q^{d+\varepsilon}}.
$$
The proof of the theorem is now complete by the  induction. $\hfill\Box$
\bigskip

\noindent{\bf Proof of Theorem \ref{maintheorem2}.~}  Suppose  that $\alpha$ is an algebraic number. Since $|\alpha|>1$,  we have $|\lambda\alpha^n|>1$ for all large enough integers $n$. Choose $\epsilon'>0$  such that   $\varepsilon'<\varepsilon\log 2/\log H(\alpha)$. Then  we get  
\begin{equation*}
0<\Vert \lambda\alpha^n+\beta\Vert<H(\alpha^n)^{-\varepsilon'}
\end{equation*}
holds true for infinitely many natural numbers $n$.  On the other hand,  by taking  $\Gamma$ to be   the subgroup  generated by $\alpha$ and  $q=1$, $\delta=\lambda$ and $u=\alpha^n$, we  see that  the hypothesis of Theorem \ref{maintheorem} is satisfied, but not the assertion, which is a contradiction.  Thus   $\alpha$ must be a transcendental number and hence the theorem. $\hfill\Box$ 
\bigskip

\noindent{\bf Acknowledgements.} I am  grateful to both the anonymous referees whose constructive  suggestions and comments helped in improving the exposition. I  express my deep gratitude to Professor Yann Bugeaud for his valuable  suggestions and comments on an  earlier version of this article. I am also very grateful to Professor Pietro Corvaja for his encouragement and Prof. R. Thangadurai for carefully reading the manuscript. This research was supported by the research grant provided by the Department of Atomic Energy, Govt. of India.


\begin{thebibliography}{9999}

\bibitem{bomb}
Bombieri, E. and Gubler, W.  \textit{Heights in Diophantine geometry}. New Mathematical Monographs, Vol. 4.  Cambridge: Cambridge University Press, 2006.
\bibitem{corv} 
Corvaja, P. and Zannier, U. ``On the rational approximation to the powers of an algebraic number: Solution of two problems of Mahler and Mendes France''. \textit{Acta Math.} 193 (2004): 175--191.
\bibitem{kul}
Kulkarni, A., Mavraki, N. M. and  Nguyen, K. D. ``Algebraic approximations to linear combinations of powers: An extension of results by Mahler and Corvaja-Zannier.''\textit{Trans. Amer. Math. Soc.} 371 (2019): 3787--3804.
\bibitem{mahler}
Mahler, K.  ``On the fractional parts of the powers of a rational number (II).'' \textit{Mathematika}, 4(1957): 122--124.

\bibitem{ridd1}
Ridout, D. ``The $p$-adic generalization of the Thue-Siegel-Roth theorem. \textit{Mathematika} 5(1958): 122--124. 

\bibitem{schmidt}
Schmidt, W. M.  \textit{Diophantine  Approximations and Diophantine Equations}.  Lecture Notes in Math. 1467. Berlin: Springer, 1991.
\bibitem{wagner}
Wagner, S. and Ziegler, V.
``Irrationality of growth constants associated with polynomial recursions. \textit{J. Integer Seq.} 24, no. 1 Article No.21.1.6 (2021): 9 pp.
\bibitem{zannier}
Zannier, U. ``\textit{Some Applications of Diophantine Approximation to Diophantine Equations} (with special emphasis on the Schmidt Subspace Theorem). Forum: Udine, 2003.
\end{thebibliography}
\end{document}